\documentclass{article}

\usepackage[T1]{fontenc}
\usepackage[utf8]{inputenc}

\usepackage[margin=1.5in]{geometry}
\usepackage{graphicx}

\usepackage{amsmath,amssymb,amsthm}
\numberwithin{equation}{section}

\usepackage[
natbib,
maxcitenames=3, 
maxbibnames=10,  
sorting=ydnt,
url=false,
doi=false,
sortcites,
defernumbers,
backref,
backend=biber
]{biblatex}
\usepackage{hyperref}

\usepackage{todonotes}
\usepackage{url}

\usepackage[nameinlink, capitalise, noabbrev]{cleveref}

\usepackage{subcaption}
\usepackage{xfrac}
\usepackage{nicefrac}
\usepackage{tikz}
\usepackage{pgfplots}
\usetikzlibrary{3d,calc}

\newtheorem{theorem}{Theorem}

\newtheorem{proposition}{Proposition}[section]
\newtheorem{lemma}{Lemma}[section]
\newtheorem{corollary}{Corollary}[section]
\theoremstyle{remark}
\newtheorem{remark}{Remark}[section]

\theoremstyle{definition}
\newtheorem{example}{Example}[section]
\newtheorem{definition}{Definition}[section]

\numberwithin{equation}{section}

\newcommand{\N}{\mathbb{N}}

\newcommand{\R}{\mathbb{R}}

\newcommand{\closure}[1]{\overline{#1}}
\newcommand{\norm}[1]{\left\Vert #1 \right\Vert}
\newcommand{\abs}[1]{\left\vert #1 \right\vert}
\DeclareMathOperator{\divtmp}{div}
\renewcommand{\div}{\divtmp}

\DeclareMathOperator{\esssup}{ess\,sup}
\DeclareMathOperator{\essinf}{ess\,inf}
\newcommand{\st}{\,:\,}
\DeclareMathOperator{\supp}{supp}
\newcommand{\dx}{\,\mathrm{d}x}
\renewcommand{\d}{\,\mathrm{d}}

\DeclareMathOperator{\sign}{sign}
\newcommand{\eps}{\varepsilon}

\DeclareMathOperator{\Lip}{Lip}
\newcommand{\KR}{\mathrm{KR}}
\newcommand{\C}{\mathrm{C}}
\renewcommand{\L}{\mathrm{L}}
\newcommand{\W}{\mathrm{W}}
\newcommand{\M}{\mathcal M}
\newcommand{\grad}{\nabla}
\newcommand{\defeq}{:=}
\DeclareMathOperator{\diam}{diam}
\newcommand{\data}{\mu}

\newcommand{\wsto}{\overset{\ast}{\rightharpoonup}}
\newcommand{\strictto}{\overset{\mathrm{str}}{\rightharpoonup}}

\newcommand{\rev}[1]{{\color{magenta}#1}}
\renewcommand{\rev}[1]{{#1}}

\newcommand{\restr}{\mathbin{\vrule height 1.6ex depth 0pt width
0.13ex\vrule height 0.13ex depth 0pt width 1.3ex}}

\makeatletter
\let\blx@rerun@biber\relax
\makeatother

\pgfplotsset{compat=1.17}

\begin{document}

\title{The inhomogeneous $p$-Laplacian equation with Neumann boundary conditions in the limit $p\to\infty$}
\author{Leon Bungert\thanks{
Hausdorff Center for Mathematics, University of Bonn, Endenicher Allee 62, Villa Maria, 53115 Bonn, Germany. \href{mailto:leon.bungert@hcm.uni-bonn.de}{leon.bungert@hcm.uni-bonn.de}
} 
}
\maketitle

\begin{abstract}
    We investigate the limiting behavior of solutions to the inhomogeneous $p$-Laplacian equation $-\Delta_p u = \data_p$ subject to Neumann boundary conditions. 
    For right hand sides which are arbitrary signed measures we show that solutions converge to a Kantorovich potential associated with the geodesic Wasserstein-1 distance. 
    In the regular case with continuous right hand sides we characterize the limit as viscosity solution to an infinity Laplacian / eikonal type equation.
    \\
    \textbf{Keywords:} $p$-Laplacian, infinity Laplacian, Wasserstein distance, optimal transport, viscosity solution
    \\
    \textbf{AMS subject classifications:} 35D30, 35D40, 49Q22
\end{abstract}

\section{Introduction}

The purpose of this paper is to study the behavior of solutions of the inhomogeneous $p$-Laplacian equation with Neumann boundary conditions as $p\to\infty$.
The precise equation we consider is
\begin{align}\label{eq:p-poisson}
\begin{cases}
    -\Delta_p u &= \data_p\quad\text{in }\Omega,\\
    \abs{\grad u}^{p-2}\frac{\partial u}{\partial\nu} &= 0,\quad\text{on }\partial\Omega,\\
    \int_\Omega \abs{u}^{p-2}u \dx &= 0,
\end{cases}
\end{align}
where $\Omega\subset\R^d$ is a Lipschitz domain and the right hand side $\data_p\in\M(\closure{\Omega})$ is a signed Radon measure which satisfies the compatibility condition $\data_p(\closure{\Omega})=0$.
We index the right hand side by $p$ in order to include the case that it varies with~$p$.
In the rest of the paper we will refer to \labelcref{eq:p-poisson} as the $p$-Poisson equation since for $p=2$ it obviously coincides with the standard Poisson equation.

We prove two convergence results, stated in \cref{sec:main_results} below.
The first one is purely variational and states that, if the right hand sides $\data_p$ converge \rev{weak-star} to a measure $\data\in\M(\closure\Omega)$ as $p\to\infty$, then weak solutions $u_p$ of \labelcref{eq:p-poisson} converge (up to a subsequence) to a Kantorovich potential $u_\infty$, which realizes the maximum in the following version of the Wasserstein-1 distance between the positive part $\data^+$ and the negative part $\data^-$ of $\data$:
\begin{align}
    \sup\left\lbrace\int_\Omega u\d\data^+ - \int_\Omega u\d\data^- \st 
     u\in\C(\closure\Omega),\;\esssup_\Omega\abs{\grad u}\leq 1
    \right\rbrace.
\end{align}
The second result uses techniques from viscosity solutions to prove that for continuous data $\data_p\in\C(\closure\Omega)$, converging uniformly to $\data\in\C(\closure\Omega)$, solutions $u_p$ converge to a viscosity solution of the following infinity Laplacian / eikonal type partial differential equation (PDE):
\begin{align}\label{eq:inf-poisson}
    \begin{cases}
    \min\left\lbrace\abs{\grad u}-1,-\Delta_\infty u\right\rbrace &= 0,\quad\text{in }\rev{\{\data>0\}},\\
    -\Delta_\infty u &= 0,\quad\text{in }\rev{\closure{\{\data\neq 0\}}^c},\\
    \max\left\lbrace1-\abs{\grad u},-\Delta_\infty u\right\rbrace &= 0,\quad\text{in }\rev{\{\data < 0\}},\\
    \max_{\closure{\Omega}} u + \min_{\closure{\Omega}} \rev{u} &=0.
    \end{cases}
\end{align}
\rev{Consequently, the only information on $\data$ which ``survives'' the limit $p\to\infty$ in the $p$-Poisson problem \labelcref{eq:p-poisson} is the support of its positive and negative part.}

Similar results have already been established for several related problems associated with the $p$-Laplace operator.
In \cite{bhattacharya1989limits}, the limit of $p$-Poisson equations with non-negative right hand side and Dirichlet boundary conditions was related to a PDE similar to \labelcref{eq:inf-poisson}.
In \cite{garcia2007neumann} the asymptotics of the homogeneous $p$-Laplacian equation with nonhomogeneous Neumann boundary conditions was investigated and related to an optimal transport problem and a viscosity PDE of infinity Laplacian type.
Furthermore, in \cite{bouchitte2003p} a vector valued version of \labelcref{eq:p-poisson} with right hand side independent of $p$ was studied.
Solutions were shown to converge to a Kantorovich potential and to solve a PDE in divergence form with measure coefficients. 
Similar results were established in \cite{mazon2014mass,evans1999differential}, however, imposing stricter regularity conditions on the right hand side in \labelcref{eq:p-poisson}.
\rev{Furthermore, in \cite{peral2009limit} the case of mixed boundary conditions and regular fixed right hand sides was related to optimal transport through a window on the boundary.}
Infinity Laplacian eigenvalue problems, their approximation with $p$-Laplacian problems, and their relation to optimal transport were investigated in \cite{juutinen1999eigenvalue,esposito2015neumann,bungert2022eigenvalue,champion2007asymptotic,champion2008infty}.

Apart from the theoretical interest in understanding the limiting behavior of solutions to \labelcref{eq:p-poisson}, our investigations are also driven by recent developments in data science.
In \cite{calder2020poisson} it was proposed to utilize the $p$-Poisson equation in order to solve semi-supervised learning tasks.
To this end, one assumes to have access to labels $g:\mathcal{O}\to\R$ of a closed subset $\mathcal{O}\subset\closure\Omega$ of the domain, in particular, $\mathcal{O}$ could be a finite collection of points.
For extending these labels from a discrete set $\mathcal{O}=\{x_i\st i=1,\dots,m\}$ with $m\in\N$ to the whole domain $\closure\Omega$ it was suggested in \cite{calder2020poisson} to solve \labelcref{eq:p-poisson} with the right hand side given by
\begin{align*}
    \data_p \equiv \data := \sum_{i=1}^m(g(x_i)-\overline{g})\delta_{x_i},\qquad \overline{g} := \frac{1}{m}\sum_{i=1}^m g(x_i),
\end{align*}
where $\delta_x\in\M(\closure\Omega)$ denotes the Dirac measure located at $x\in\closure\Omega$.
While this method, termed ``Poisson learning'', performs very well in practice, a full analysis is still pending. 
In particular, a rigorous convergence proof of the finite-dimensional approximation of Poisson learning on weighted graphs---which is used in applications---would be desirable.

The results of the present article apply to the continuum description of Poisson learning and, in particular, address the asymptotics as $p\to\infty$.
For the balanced case of two labelled classes with equal size, i.e., $g:\mathcal{O}\to\{\pm 1\}$ and $\overline{g}=0$, our main results can be interpreted as follows: 
The labelling function $u$ arising as limit of solutions to Poisson learning as $p\to\infty$ is directly connected to the solution of the optimal transport problem, which transports the empirical measure $\sum_{i\st g(x_i)=+1}\delta_{x_i}$ of the points with label $+1$ to the empirical measure $\sum_{i\st g(x_i)=-1}\delta_{x_i}$ of the points with label $-1$.

The plan of this paper is the following:
\cref{sec:prelim} reviews some important mathematical background and states our main results which are proved in \cref{sec:limit}.
In more detail, \cref{sec:cvgc} proves compactness of solutions of \labelcref{eq:p-poisson} as $p\to\infty$, \cref{sec:OT} is devoted to the optimal transport characterization of cluster points, and \cref{sec:PDE} relates them to the limiting PDE \labelcref{eq:inf-poisson}.

\section{Mathematical preliminaries and main results}
\label{sec:prelim}

\subsection{Weak solution to the \texorpdfstring{$p$}{p}-Laplacian equation}

The $p$-Laplacian for $p\in[1,\infty)$ is defined as
\begin{align}
\Delta_p u:=\div(\abs{\grad u}^{p-2}\grad u).
\end{align}
For $\C^2$-functions $u$ it admits the decomposition formula
\begin{align}\label{eq:decomposition}
    \Delta_p u = \abs{\grad u}^{p-2}\left(\Delta u + (p-2)\frac{\Delta_\infty u}{\abs{\grad u}^2}\right),
\end{align}
where $\Delta u = \div(\grad u)$ denotes the Laplacian and $\Delta_\infty u \defeq \langle\grad u,D^2u\grad u\rangle$ is called the infinity Laplacian.

Since we are interested in the case $p\to\infty$ anyway, we assume in the whole article that $p>d$, in which case the Sobolev embedding $\W^{1,p}(\Omega)\hookrightarrow\C^{0,1-\frac{d}{p}}(\closure\Omega)$ makes sure that the following concept of weak solutions to \labelcref{eq:p-poisson} is well-defined.

\begin{definition}\label{def:weak_solutions}
Let $p>d$. A function $u\in\W^{1,p}(\Omega)$ is called a weak solution to \labelcref{eq:p-poisson} if it satisfies $\int_\Omega\abs{u}^{p-2}u\dx = 0$ and
\begin{align}\label{eq:weak-formulation}
    \int_\Omega \abs{\grad u}^{p-2}\grad u\cdot\grad\phi\,\dx = \int_\Omega \phi\d\data_p,\quad\forall\phi\in\W^{1,p}(\Omega).
\end{align}
\end{definition}
It is obvious that weak solutions in the sense of \cref{def:weak_solutions} coincide with solutions of the variational problem
\begin{align}\label{eq:varprob_p}
    \min\left\lbrace \frac{1}{p}\int_\Omega\abs{\grad u}^p \d x - \int_\Omega u \d\data_p \st u\in\W^{1,p}(\Omega),\;\int_\Omega\abs{u}^{p-2}u\d x=0 \right\rbrace
\end{align}
since the Euler-Lagrange equations of this problem precisely coincide with \labelcref{eq:weak-formulation}, cf. \cite{lindqvist2017notes}.
Using standard arguments from calculus of variations it can be shown that this problem admits a unique solution for every $p>1$.
Apart from guaranteeing existence and uniqueness, this variational characterization will be essential for deriving the optimal transport characterization of the limit $\lim_{p\to\infty}u_p$ of weak solutions $u_p\in\W^{1,p}(\Omega)$.
For higher regularity statements for solutions of the $p$-Poisson equation we refer the interested reader to \cite{lindgren2017regularity}.

\subsection{Geodesic geometry}

As it turns out, the correct metric on $\Omega$ when working with \labelcref{eq:p-poisson} (or \labelcref{eq:varprob_p}) and its limit as $p\to\infty$ is not the Euclidean one but the \emph{geodesic distance}.
It is defined as
\begin{align}\label{eq:geo_dist}
    d_\Omega(x,y) := \inf 
    \left\lbrace 
    \int_0^1 \abs{\dot\gamma(t)}\d t \st \gamma\in C^1([0,1],\Omega),\,\gamma(0)=x,\,\gamma(1)=y
    \right\rbrace
\end{align}
and turns $(\Omega,d_\Omega)$ into a length space.
The geodesic distance measures the length of the shortest curve in $\Omega$ connecting two points.
If $\Omega$ is convex the curve $\gamma(t)=(1-t)x+ty$ shows $d_\Omega(x,y)=\abs{x-y}$ but in general it holds $d_\Omega(x,y)\geq\abs{x-y}$.
A derived quantity which appears naturally in the context of the Neumann problem \labelcref{eq:p-poisson} is the \emph{geodesic diameter} of $\Omega$, defined as
\begin{align}
    \diam(\Omega) := \sup_{x,y\in\Omega}d_\Omega(x,y).
\end{align}
The geodesic diameter appears as optimal constant for in the inequality
\begin{align}\label{ineq:embedding}
    \esssup_\Omega\abs{u} \leq
    \frac{\diam(\Omega)}{2}\esssup_\Omega\abs{\grad u},\quad\forall u\in\W^{1,\infty}(\Omega)\st\esssup_{\Omega}u+\essinf_{\Omega}u=0
\end{align}
and as first non-trivial Neumann eigenvalue of the infinity Laplacian \cite{rossi2016first,esposito2015neumann}, given by
\begin{align}
    \lambda_\infty := \frac{2}{\diam(\Omega)}.
\end{align}
One can use the geodesic distance to define the geodesic Lipschitz constant of $u\in\C(\closure\Omega)$ as
\begin{align}\label{eq:geodesic_Lip_const}
    \Lip_\Omega(u) := \sup_{\substack{x,y\in\Omega\\x\neq y}}\frac{\abs{ u(x)- u(y)}}{d_\Omega(x,y)}.
\end{align}
With this at hand one can introduce a geodesic version of the Wasserstein-1 distance:
\begin{align}\label{eq:geodesic_wasserstein_distance}
    W_{1,\Omega}(\data^+,\data^-) := \sup\left\lbrace\int_\Omega u\d\data^+ - \int_\Omega u\d\data^- \st 
     u\in\C(\closure\Omega),\;\Lip_\Omega(u)\leq 1
    \right\rbrace.
\end{align}
Note that, as stated in \cite[page 269]{brezis2011functional}, any function $ u\in\W^{1,\infty}(\Omega)$ has a continuous representative (denoted by the same symbol) and it holds
\begin{align}
    \abs{ u(x)- u(y)}\leq \esssup_\Omega\abs{\grad u} d_\Omega(x,y),\quad\forall x,y\in\Omega.
\end{align}
This shows that $\Lip_\Omega(u)\leq\esssup_\Omega\abs{\grad u}$.
Furthermore, since for points $x,y$ that lie in a ball which is fully contained in $\Omega$ it holds $d_\Omega(x,y)=\abs{x-y}$, it is easily seen (see \cite[page 23]{bhattacharya1989limits}) that in fact
\begin{align}\label{eq:geodesic_Lip}
    \Lip_\Omega(u)=\esssup_\Omega\abs{\grad u}.
\end{align}

\subsection{\rev{Weak-star} convergence of measures}

For measuring the convergence of the right hand side measures $\data_p$ in \labelcref{eq:p-poisson} as $p\to\infty$ we \rev{utilize weak-star} convergence of measures.

\begin{definition}[\rev{Weak-star} convergence of measures]
As sequence of Radon measures $(\data_n)_{n\in\N}\subset\M(\closure{\Omega})$ is said to converge \rev{weak-star} to $\data\in\M(\closure{\Omega})$ (written $\data_n\wsto\data$) if
\begin{align}
    \lim_{n\to\infty}\int_\Omega u \d\data_n =\int_\Omega u\d\data\quad\forall u\in \C(\closure{\Omega}).
    %\quad\text{and}\quad \limsup_{n\to\infty}\abs{\data_n}(\closure\Omega)\leq \abs{\data}(\closure\Omega).
\end{align}
\end{definition}
\begin{remark}[Smooth approximation]
It is easy to see that any Radon measure can be approximated in the \rev{weak-star} topology by convolving it with a mollifier.
\end{remark}
% \begin{remark}[Mass preservation]
% Using the lower semi-continuity of the total variation norm $\abs{\cdot}(\closure{\Omega})$ shows that a strictly converging sequence $\data_n\strictto\data$ satisfies
% \begin{align}
%     \abs{\data}(\closure{\Omega})=\lim_{n\to\infty}\abs{\data_n}(\closure\Omega),
% \end{align}
% where for weak-star convergent measures only $\leq$ is true.
% Hence, strict convergence makes sure that no mass gets lost in the limit which will be important for our optimal transport characterization.
% \end{remark}

\subsection{Main results}
\label{sec:main_results}

The following are our main results. 
The proof of \cref{thm:wasserstein} can be found in \cref{sec:OT} and the one of \cref{thm:viscosity_PDE}, along with precise definitions of the notion of viscosity solutions and some corollaries, in \cref{sec:PDE}.

\begin{theorem}\label{thm:wasserstein}
Assume that $\data_p\wsto\data$ in $\M(\closure{\Omega})$ as $p\to\infty$ and let $\data=\data^+ - \data^-$, with non-negative measures $\data^\pm\in\M(\closure{\Omega})$, be the Jordan decomposition of $\data$.
Then (up to a subsequence) weak solutions $u_p\in\W^{1,p}(\Omega)$ of \labelcref{eq:p-poisson} uniformly converge to a function $u_\infty\in\W^{1,\infty}(\Omega)$ which satisfies
\begin{align*}
    W_{1,\Omega}(\data^+,\data^-) = \int_\Omega u_\infty\d\data^+ - \int_\Omega u_\infty\d\data^-.
\end{align*}
\end{theorem}

\begin{theorem}\label{thm:viscosity_PDE}
If $\data_p\in\C(\closure{\Omega})$ converges uniformly to $\data\in\C(\closure{\Omega})$ as $p\to\infty$, then (up to a subsequence) weak solutions $u_p\in\W^{1,p}(\Omega)$ of \labelcref{eq:p-poisson} converge uniformly to a function $u_\infty\in\W^{1,\infty}(\Omega)$ which is a viscosity solution of 
\begin{align}\label{eq:Omega_pmz}
    \begin{cases}
    \min\left\lbrace\abs{\grad u}-1,-\Delta_\infty u\right\rbrace &= 0,\quad\text{in }\rev{\rev{\{\data>0\}}},\\
    -\Delta_\infty u &= 0,\quad\text{in }\rev{\rev{\closure{\{\data\neq 0\}}^c}},\\
    \max\left\lbrace1-\abs{\grad u},-\Delta_\infty u\right\rbrace &= 0,\quad\text{in }\rev{\rev{\{\data<0\}}},\\
    \max_{\closure{\Omega}} u + \min_{\closure{\Omega}} \rev{u} &=0.
    \end{cases}    
\end{align}
\end{theorem}

\section{Limit behavior as \texorpdfstring{$p\to\infty$}{p to infinity}}
\label{sec:limit}

\subsection{Convergence of solutions}
\label{sec:cvgc}

In this section we show that if the sequence of right hand sides $\data_p$ in \labelcref{eq:p-poisson} has uniformly bounded mass, then the sequence of solutions $(u_p)_{p>0}$ admits a convergent subsequence.

To this end we first derive an upper bound for the $p$-Dirichlet energy $\int_\Omega\abs{\grad u}^p\dx$ in terms of the data which will then allow us to deduce convergence.

\begin{proposition}\label{prop:limsup_ineq}
Let $u_p\in\W^{1,p}(\Omega)$ be a weak solution of \labelcref{eq:p-poisson} with data $\data_p\in\M(\closure{\Omega})$.
Then it holds
\begin{align*}
    \limsup_{p\to\infty}\left(\int_\Omega\abs{\grad u_p}^p\dx\right)^{1-\frac{1}{p}} \leq \frac{\diam(\Omega)}2 \limsup_{p\to\infty}\abs{\data_p}(\closure{\Omega}).
\end{align*}
\end{proposition}
\begin{proof}
Choosing $\phi=u_p$ in \labelcref{eq:weak-formulation} and using Hölder's and Morrey's inequalities yield
\begin{align*}
    \int_\Omega \abs{\grad u_p}^p\dx 
    &= 
    \int_\Omega u_p\d\data_p
    %\\
    \leq
    \abs{\data_p}(\closure{\Omega})\esssup_\Omega\abs{u_p}
    \\
    &\leq \frac{1}{\sqrt[p]{\sigma_p}}\abs{\data_p}(\closure{\Omega}) \left(\int_\Omega\abs{\grad u_p}^p\dx\right)^\frac{1}{p},
\end{align*}
where the optimal constant for the Morrey inequality is defined as
\begin{align*}
    \sigma_p := 
    \inf\left\lbrace\frac{\int_\Omega\abs{\grad u}^p\dx}{\esssup_\Omega\abs{u}^p}\st u\in\W^{1,p}(\Omega),\,\int_\Omega\abs{u}^{p-2}u\dx = 0\right\rbrace.
\end{align*}
Using that $p\mapsto\sqrt[p]{\sigma_p}$ converges to the value $\frac{2}{\diam(\Omega)}\in(0,\infty)$, which is the first non-trivial Neumann eigenvalue of the infinity Laplacian \cite{rossi2016first,esposito2015neumann}, concludes the proof.
\end{proof}

Before proving the convergence theorem we need the following important lemma.

\begin{lemma}\label{lem:convergence_norms}
Let $u_p\subset \L^p(\Omega)$ converge to $u_\infty\in L^\infty(\Omega)$ uniformly on $\Omega$.
Then for every $0\leq k \leq p-1$ it holds
\begin{align}
    \lim_{p\to\infty}\left(\int_\Omega\abs{u_p}^{p-k}\d x\right)^\frac{1}{p-k} = \esssup_\Omega\abs{u_\infty}.
\end{align}
\end{lemma}
\begin{proof}
Let $\varepsilon>0$ be given. 
Then for $p$ sufficiently large it holds $\esssup\abs{u_p-u_\infty}\leq \varepsilon$.
Consequently, using Minkowski's inequality
\begin{align*}
    \left(\int_\Omega\abs{u_p}^{p-k}\d x\right)^\frac{1}{p-k} 
    &\leq 
    \left(\int_\Omega\abs{u_p-u_\infty}^{p-k}\d x\right)^\frac{1}{p-k} + \left(\int_\Omega\abs{u_\infty}^{p-k}\d x\right)^\frac{1}{p-k} \\
    &\leq  
    \eps \abs{\Omega}^\frac{1}{p-k} + \left(\int_\Omega\abs{u_\infty}^{p-k}\d x\right)^\frac{1}{p-k} 
\end{align*}
and hence
\begin{align*}
    \limsup_{p\to\infty}\left(\int_\Omega\abs{u_p}^{p-k}\d x\right)^\frac{1}{p-k} 
    \leq
    \varepsilon + \esssup_\Omega\abs{u_\infty}.
\end{align*}
Using the reverse triangle inequality one analogously obtains
\begin{align*}
    \liminf_{p\to\infty}\left(\int_\Omega\abs{u_p}^{p-k}\d x\right)^\frac{1}{p-k} 
    \geq
    -\varepsilon + \esssup_\Omega\abs{u_\infty}.
\end{align*}    
Combining these two inequalities and using that $\varepsilon>0$ was arbitrary concludes the proof.
\end{proof}

Now we can prove that the sequence of solutions of \labelcref{eq:p-poisson} has a convergent subsequence.

\begin{proposition}\label{prop:cvgc_p_to_infty}
Let $u_{p}\in\W^{1,p}(\Omega)$ be a weak solution of \labelcref{eq:p-poisson} with data $\data_p\in\M(\closure{\Omega})$ and assume that the data satisfies 
\begin{align*}
    \limsup_{p\to\infty}\abs{\data_p}(\closure{\Omega})<\infty.
\end{align*}
Then there exists a function $u_\infty\in\W^{1,\infty}(\Omega)$ such that as $p\to\infty$ (up to a subsequence) the functions $u_p$ converge to $u_\infty$ uniformly and weakly in $\W^{1,m}(\Omega)$ for any $m>1$.
Furthermore, it holds
\begin{align}\label{ineq:norm_bounds}
    \esssup_\Omega\abs{u_\infty} \leq  \frac{\diam(\Omega)}{2},\qquad\esssup_\Omega\abs{\grad u_\infty} \leq 1.
\end{align}
\end{proposition}

\begin{proof}
We follow the strategy from \cite{bhattacharya1989limits}.
For $p>m$ Hölder's inequality yields
\begin{align*}
    \int_\Omega\abs{\grad u_p}^m\d x 
    \leq
    \left(\int_\Omega\abs{\grad u_p}^{p}\dx\right)^\frac{m}{p}\abs{\Omega}^{1-\frac{m}{p}}.
\end{align*}
Consequently, using \cref{prop:limsup_ineq} it follows
\begin{align}\label{ineq:cv_pf_1}
    \limsup_{p\to\infty}\int_\Omega\abs{\grad u_p}^m\d x \leq \abs{\Omega} < \infty.
\end{align}
Introducing the first non-zero eigenvalue of the $p$-Laplacian operator \cite{rossi2016first,esposito2015neumann}
\begin{align}\label{eq:eigenvalue}
    \lambda_p := \inf\left\lbrace\frac{\int_\Omega\abs{\grad u}^p\dx}{\int_\Omega\abs{u}^p\dx}\st u\in\W^{1,p}(\Omega),\,\int_\Omega\abs{u}^{p-2}u\dx = 0\right\rbrace
\end{align}
it holds
\begin{align}\label{ineq:p-norm-bounds}
    \int_\Omega \abs{u_p}^{p}\d x \leq \frac{1}{\lambda_p} \int_\Omega\abs{\grad u_p}^{p}\d x 
\end{align}
and therefore we can estimate
\begin{align*}
    \int_\Omega\abs{u_p}^m\d x
    \leq
    \abs{\Omega}^{1-\frac{m}{p}}
    \left(\int_\Omega\abs{u_p}^p\dx\right)^\frac{m}{p} 
    \leq 
    \abs{\Omega}^{1-\frac{m}{p}} 
    \frac{1}{\sqrt[p]{\lambda_p}^m}
    \left(\int_\Omega\abs{\grad u_p}^p\dx\right)^\frac{m}{p}.
\end{align*}
Using \cref{prop:limsup_ineq} together with the fact that according to \cite{esposito2015neumann} it holds $\sqrt[p]{\lambda_p}\to\lambda_\infty:=\frac{2}{\diam(\Omega)}\in(0,\infty)$ as $p\to\infty$ we obtain
\begin{align}\label{ineq:cv_pf_2}
    \limsup_{p\to\infty}\int_\Omega\abs{u_p}^m \d x
    \leq\frac{\abs{\Omega}}{\lambda_\infty^m}<\infty.
\end{align}
Thanks to \labelcref{ineq:cv_pf_1,ineq:cv_pf_2} the sequence $u_p$ has uniformly bounded $\W^{1,m}$-norms and hence (up to a subsequence) converges weakly to a function $u_\infty$ in $\W^{1,m}(\Omega)$.
Furthermore, for $m>d$ one has the compact embedding \cite{ziemer2012weakly} of $\W^{1,m}(\Omega)$ into $\C^{0,1-\frac{d}{m}}(\closure{\Omega})$ which (after another round of subsequence refinement) proves the uniform convergence.

It remains to argue that $u_\infty\in\W^{1,\infty}(\Omega)$ and to prove \labelcref{ineq:norm_bounds}.
Using the weak lower semi-continuity of the $\L^m$-norm  we obtain from \labelcref{ineq:cv_pf_1} that
\begin{align*}
    \int_\Omega\abs{\grad u_\infty}^m \d x\leq
    \liminf_{p\to\infty}\int_\Omega\abs{\grad u_p}^m\d x \leq \abs{\Omega}.
\end{align*}
Taking the $m$-th root and sending $m\to\infty$ an application of \cref{lem:convergence_norms} shows
\begin{align*}
    \esssup_\Omega\abs{\grad u_\infty} 
    = 
    \lim_{m\to\infty}\left(\int_\Omega\abs{\grad u_\infty}^m\d x\right)^\frac{1}{m}
    \leq
    1.
\end{align*}
Using again lower semi-continuity and \labelcref{ineq:cv_pf_2} yields
\begin{align*}
    \int_\Omega\abs{u_\infty}^m\d x\leq\liminf_{p\to\infty}\int_\Omega\abs{u_p}^m\d x \leq \frac{\abs{\Omega}}{\lambda_\infty^m}.
\end{align*}
Taking the $m$-th root and applying \cref{lem:convergence_norms} with $p=m$ and $k=0$ yields
\begin{align*}
    \esssup_\Omega\abs{u_\infty} 
    &= 
    \lim_{m\to\infty}\left(\int_\Omega\abs{u_\infty}^m\dx\right)^\frac{1}{m}
    \leq 
    \limsup_{m\to\infty}
    \frac{\abs{\Omega}^\frac{1}{m}}{\lambda_\infty} = \frac{1}{\lambda_\infty} = \frac{\diam(\Omega)}{2}.
\end{align*}
Hence, we have established all inequalities in \labelcref{ineq:norm_bounds}.
\end{proof}

\begin{example}\label{ex:p-laplacians}
If $\d\data_p(x)=\sign(x)\d x$ for all $p>1$, the explicit solution of the $p$-Laplacian equation \labelcref{eq:p-poisson} is given by $u_p(x)=\frac{p-1}{p}\sign(x)\left[1-(1-\abs{x})^\frac{p}{p-1}\right]$.
As $p\to\infty$ the functions $u_p$ converge uniformly to $u_\infty(x)=x$, see \cref{fig:1d_sols_p}.
Note that the Neumann boundary conditions get lost in the limit $p\to\infty$, see also \cref{sec:PDE}.

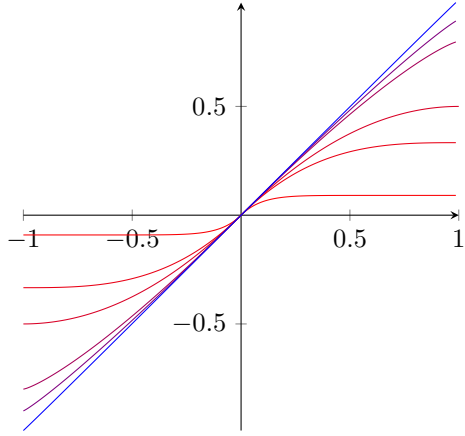
\begin{figure}[thb]
\centering
\begin{tikzpicture}
    \begin{axis}[
axis lines = center,
axis equal image,
xtick = {-1,-0.5,0,0.5,1},
xticklabels = {$-1$,$-0.5$,$0$,$0.5$,$1$},
ytick = {-1,-0.5,0,0.5,1},
yticklabels = {$-1$,$-0.5$,$0$,$0.5$,$1$}
]

\def\function(#1,#2){
((#2)-1)/(#2)*sign(#1)*(1-(1-abs(#1))^((#2)/((#2)-1)))
}

\foreach [evaluate=\p as \n using ln(\p)/ln(100) * 100] \p in {1.1,1.5,2,5,10,100}
{
\edef\temp{%
\noexpand
\addplot [
    domain=-1:1,
    samples=151, 
    color=blue!\n!red,
    smooth
]
{\function(x,\p)};
}\temp
}
\end{axis}
\end{tikzpicture}
\caption{Solutions $u_p$ of \labelcref{eq:p-poisson} for $p\in\{1.1,1.5,2,5,10,100\}$ (from red to blue).}
\label{fig:1d_sols_p}
\end{figure}

\end{example}

\subsection{Optimal transport characterization}
\label{sec:OT}

The main theorem in this section characterizes the limit $u_\infty$ as optimal transport potential.
We assume that the data measures $\data_p$ converge in the \rev{weak-star sense} of measures.
\rev{This} makes sure that one can pass to the limit in duality products where both factors converge, as the following lemma shows.
\begin{lemma}\label{lem:duality}
If $\data_n\wsto\data$ in $\M(\closure{\Omega})$ and $u_n\to u$ uniformly in $\C(\closure{\Omega})$ then it holds
\begin{align}
    \lim_{n\to\infty}\int_\Omega u_n\d\data_n = \int_\Omega u\d\data.
\end{align}
\end{lemma}
\begin{proof}
With the abbreviation $\langle\data,u\rangle:=\int_\Omega u\d\data$ we can compute
\begin{align*}
    \abs{\langle\data_n,u_n\rangle-\langle\data,u\rangle} 
    &= 
    \abs{\langle\data_n,u_n-u\rangle+\langle\data_n,u\rangle-\langle\data,u\rangle} 
    \\
    &\leq 
    \abs{\data_n}(\closure{\Omega})\sup_{\closure{\Omega}}\abs{u_n-u} + \abs{\langle\data_n-\data,u\rangle}.
\end{align*}
\rev{The Banach--Steinhaus theorem (or uniform boundedness principle) \cite[Section 2.2]{brezis2011functional} makes sure that $\sup_n\abs{\data_n}(\closure{\Omega})<\infty$.} 
\rev{Together with the} uniform convergence \rev{of $u_n$} \rev{and the weak-star convergence of $\data_n$ this implies} the right hand side goes to zero when taking the $\limsup$ as $n\to\infty$.
\end{proof}
To set the scene for the optimal transport characterization we remind the reader of the usual Wasserstein-$1$ distance $W_1(\data^+,\data^-)$ of the two measures $\data^\pm$, defined as
\begin{align}\label{eq:wasserstein_distance}
    W_1(\data^+,\data^-) := \sup\left\lbrace\int_\Omega u\d\data^+ - \int_\Omega u\d\data^- \st 
     u\in\C(\closure\Omega),\;\Lip( u)\leq 1
    \right\rbrace,
\end{align}
where the Lipschitz constant $\Lip( u)$ in \labelcref{eq:wasserstein_distance} is
\begin{align}
    \Lip( u) := \sup_{\substack{x,y\in\Omega\\x\neq y}}\frac{\abs{ u(x)- u(y)}}{\abs{x-y}}.
\end{align}
Functions $ u\in\C(\closure\Omega)$ which attain the supremum in \labelcref{eq:wasserstein_distance} are typically referred to as Kantorovich potentials.

The Lipschitz constant, and hence also the Wasserstein-1 distance, is defined with respect to the Euclidean metric on $\R^d$.
This is, however, not the most natural metric to consider on the (possibly non-convex) domain $\Omega$.
Indeed it can happen that two points in $\Omega$ have a small Euclidean distance although transporting two measure concentrated on these points onto each other within $\Omega$ requires a long transportation path.
To overcome this one can use the geodesic distance on $\Omega$, defined in \labelcref{eq:geo_dist}. 
Correspondingly, one can also introduce the geodesic Lipschitz constant \labelcref{eq:geodesic_Lip_const} and geodesic Wasserstein-1 distance \labelcref{eq:geodesic_wasserstein_distance}.

As \cref{thm:wasserstein} states, this geodesic transport distance \labelcref{eq:geodesic_wasserstein_distance} arises naturally in the limiting problem of the $p$-Poisson equation \labelcref{eq:p-poisson}.
We now give the proof of this statement.

\begin{proof}[Proof of \cref{thm:wasserstein}]
First we note that the \rev{weak-star} convergence of $\data_p$ \rev{together with the Banach--Steinhaus theorem} in particular implies that
\begin{align*}
\limsup_{p\to\infty}\abs{\data_p}(\closure{\Omega})<\infty
\end{align*}
such that \cref{prop:cvgc_p_to_infty} assures the existence of a subsequential uniform limit $u_\infty$.

Let $ u\in\W^{1,\infty}(\Omega)$ with $\esssup_\Omega\abs{\grad u}\leq 1$ be arbitrary.
Without loss of generality we can assume that $\int_\Omega\abs{ u}^{p-2} u\d x=0$.
Since $u_p$ in particular solves \labelcref{eq:varprob_p} it holds
\begin{align*}
    \frac{1}{p}\int_\Omega\abs{\grad u_p}^p \d x - \int_\Omega u_p \d\data_p \leq \frac{1}{p}\int_\Omega\abs{\grad  u}^p \d x - \int_\Omega  u \d\data_p.
\end{align*}
We can rearrange this inequality to
\begin{align*}
    &\frac{1}{p}\int_\Omega\abs{\grad u_p}^p \d x -
    \frac{1}{p}\int_\Omega\abs{\grad  u}^p \d x + 
    \int_\Omega  u \d\data_p^+ - \int_\Omega  u \d\data_p^- \\
    \leq &\int_\Omega u_p \d\data_p^+ - \int_\Omega u_p \d\data_p^-,
\end{align*}
where $\data_p=\data_p^+ - \data_p^-$, with non-negative measures $\data_p^\pm\in\M(\closure{\Omega})$, is the Jordan decomposition of $\data_p$.
Obviously it holds $\data_p^\pm\wsto\data^\pm$ as $p\to\infty$ since the measures $\data_p^\pm$ are mutually singular.
Now we use \cref{lem:duality} together with the fact that the first term is non-negative and $\abs{\grad u}\leq 1$ a.e. in $\Omega$ to obtain
\begin{align*}
    \int_\Omega  u \d\data^+ - \int_\Omega  u \d \data^-
    \leq \int_\Omega u_\infty \d \data^+ - \int_\Omega  u_\infty \d\data^-.
\end{align*}
Since by \labelcref{ineq:norm_bounds,eq:geodesic_Lip} the function $u_\infty$ is feasible for the optimization problem in \labelcref{eq:geodesic_wasserstein_distance}, taking the supremum over $ u$ shows the assertion.
\end{proof}

Since according to \cref{prop:cvgc_p_to_infty} the limit $u_\infty$ also satisfies $\esssup_\Omega\abs{u_\infty}\leq\frac{\diam(\Omega)}{2}$, one could also have the idea to include a boundedness condition in the optimization problem in \labelcref{eq:geodesic_wasserstein_distance}.
This is motivated by the so-called Kantorovich--Rubinstein (KR) norm of the measure $\data=\data^+-\data^-$ on the length space $(\Omega,d_\Omega)$, which is defined as
\begin{align}\label{eq:kantorovich-rubinstein}
    \norm{\data}_{\KR(\closure{\Omega})} \defeq
    \sup\left\lbrace\int_\Omega  u \d\data \st  u\in\C(\closure\Omega),\;\esssup_\Omega\abs{ u}\leq 1 ,\;\Lip_\Omega(u)\leq 1 \right\rbrace.
\end{align}
The reason why the KR norm does not appear naturally in our context is that for measures with zero mass it is equivalent (and for suitably scaled domains even equal) to the so-called dual Lipschitz norm.
This norm coincides with the geodesic Wasserstein distance of the positive and negative part of the measure and is defined as
\begin{align}\label{eq:dual_lipschitz}
    \norm{\data}_{\Lip_\Omega^\ast(\closure{\Omega})} \defeq
    \sup\left\lbrace\int_\Omega  u \d\data \st  u\in\C(\closure\Omega),\;\Lip_\Omega(u)\leq 1 \right\rbrace
    =
    W_{1,\Omega}(\data^+,\data^-).
\end{align}
For completeness, the equivalence is stated in the following proposition.
\begin{proposition}
Let $\data\in\M(\closure{\Omega})$ satisfy $\data(\closure{\Omega})=0$.
Then it holds
\begin{align}
    \norm{\data}_{\KR(\closure{\Omega})} 
    \leq 
    \norm{\data}_{\Lip_\Omega^\ast(\closure{\Omega})}
    \leq
    \left(1\vee\frac{\diam(\Omega)}2\right)
    \norm{\data}_{\KR(\closure{\Omega})}. 
\end{align}
\end{proposition}
\begin{proof}
The proof works just as in \cite[Proposition 4.3]{bungert2022eigenvalue}, see also \cite[Lemma 2.1]{lellmann2014imaging}.
By omitting the constraint $\esssup_\Omega\abs{ u}\leq 1$ we obtain the first inequality $\norm{\data}_{\KR(\closure\Omega)}\leq\norm{\data}_{\Lip_\Omega^\ast(\closure\Omega)}$.

For the other inequality we argue as follows:
Since $\data$ has zero mass, we can without loss of generality assume that the supremum in \labelcref{eq:dual_lipschitz} is taken over functions that satisfy $\esssup_\Omega u +\essinf_\Omega u = 0$ by replacing $ u$ with $u-c$ for a suitable constant.
Then, using \labelcref{ineq:embedding} we get for all $ u\in\C(\closure\Omega)$ with $\Lip_\Omega(u)\leq 1$ that
\begin{align*}
    \esssup_\Omega\abs{ u} \leq \frac{\diam(\Omega)}{2}\underbrace{\esssup_\Omega\abs{\grad u}}_{=\Lip_\Omega(u)} \leq \frac{\diam(\Omega)}{2}.
\end{align*}
Letting $t:=1\vee\frac{\diam(\Omega)}{2}\geq 1$ it holds
\begin{align*}
    \norm{\data}_{\Lip_\Omega^\ast(\closure\Omega)}
    &=
    t
    \sup\left\lbrace\int_\Omega \frac{1}{t} u \d\data \st  u\in\C(\closure\Omega),\;\Lip_\Omega(u)\leq 1 \right\rbrace
    \\
    &=
    t
    \sup\left\lbrace\int_\Omega \frac{1}{t} u \d\data \st  u\in\C(\closure\Omega),\;\esssup_\Omega\abs{ u}\leq t,\;\Lip_\Omega(u)\leq 1 \right\rbrace
    \\
    &=
    t
    \sup\left\lbrace\int_\Omega  u \d\data \st  u\in\C(\closure\Omega),\;\esssup_\Omega\abs{ u}\leq 1,\;\Lip_\Omega(u)\leq \frac{1}{t} \right\rbrace
    \\
    &\leq 
    t
    \sup\left\lbrace\int_\Omega  u \d\data \st  u\in\C(\closure\Omega),\;\esssup_\Omega\abs{ u}\leq 1,\;\Lip_\Omega(u)\leq 1\right\rbrace
    \\
    &=
    t\norm{\data}_{\KR(\closure\Omega)}.
\end{align*}
\end{proof}

\subsection{PDE characterization}
\label{sec:PDE}

Now we also give a PDE characterization of the limit $u_\infty$, which we have shown to be a Kantorovich potential in the previous section.
Note that Kantorovich potentials are typically not unique which is why it is interesting to verify that the limiting procedure $p\to\infty$ selects a more regular potential.
Since $u_\infty$ turns out to solve an infinity Laplacian type PDE in the viscosity sense, we also have to work with viscosity solutions for finite~$p$.
However, for that we have to assume that the data $\data_p$ are continuous \rev{and converge uniformly}.

Let us first define what it means to be a viscosity solution to the $p$-Poisson equation~\labelcref{eq:p-poisson}.
In particular, one has to interpret the Neumann boundary conditions in the viscosity sense, see also~\cite{esposito2015neumann,amato2021solutions}.
\rev{
As explained in \cite{crandall1992user} the proper way to understand boundary conditions for boundary value problems of the form
\begin{align*}
    \begin{cases}
    F(x,u(x),\grad u(x), D^2 u(x))=0\quad&\text{in }\Omega,\\
    B(x,u(x),\grad u(x), D^2 u(x))=0\quad&\text{on }\partial\Omega,
    \end{cases}
\end{align*}
is to demand that subsolutions satisfy 
\begin{align*}
    \min\left\lbrace F(x,u(x),\grad u(x), D^2 u(x)), B(x,u(x),\grad u(x), D^2 u(x))\right\rbrace\leq 0 \quad\text{on $\partial\Omega$}
\end{align*}
and supersolutions satisfy the converse inequality with a $\max$ in place of the $\min$.
}
\begin{definition}[Viscosity solutions of the $p$-Poisson equation]\label{def:visc_p}
Let $\data_p\in\C(\closure\Omega)$.
An upper semi-continuous function $u:\closure\Omega\to\R$ is called viscosity subsolution of \labelcref{eq:p-poisson} if 
\begin{itemize}
    \item for all $x_0\in\Omega$ and $\phi\in C^2(\Omega)$ such that $u-\phi$ has a local maximum at $x_0$ it holds
    \begin{align*}
        -\Delta_p \phi(x_0)-\data_p(x_0) \leq 0,
    \end{align*}
    \item for all $x_0\in\partial\Omega$ and $\phi\in C^2(\closure\Omega)$ such that $u-\phi$ has a local maximum at $x_0$ it holds
    \begin{align*}
        \min\left\lbrace \abs{\grad\phi(x_0)}^{p-2}\frac{\partial\phi}{\partial\nu}(x_0), -\Delta_p\phi(x_0) - \data_p(x_0) \right\rbrace \leq 0,
    \end{align*}
    \item it holds $\int_\Omega\abs{u}^{p-2}u\d x\leq 0$.
\end{itemize}
A lower semi-continuous function $u:\closure\Omega\to\R$ is called viscosity supersolution of \labelcref{eq:p-poisson} if 
\begin{itemize}
    \item for all $x_0\in\Omega$ and $\phi\in C^2(\Omega)$ such that $u-\phi$ has a local minimum at $x_0$ it holds
    \begin{align*}
        -\Delta_p \phi(x_0)-\data_p(x_0) \geq 0,
    \end{align*}
    \item for all $x_0\in\partial\Omega$ and $\phi\in C^2(\closure\Omega)$ such that $u-\phi$ has a local maximum at $x_0$ it holds
    \begin{align*}
        \max\left\lbrace \abs{\grad\phi(x_0)}^{p-2}\frac{\partial\phi}{\partial\nu}(x_0), -\Delta_p\phi(x_0) - \data_p(x_0) \right\rbrace \geq 0,
    \end{align*}
    \item it holds $\int_\Omega\abs{u}^{p-2}u\d x \geq 0$.
\end{itemize}
A function $u\in\C(\closure\Omega)$ is called viscosity solution of \labelcref{eq:p-poisson} if it is both a sub- and a supersolution.
\end{definition}
We need the following well-known statement which asserts that weak solutions to the $p$-Poisson equation are also viscosity solutions.
\begin{lemma}
Let $\data_p\in \C(\closure{\Omega})$.
\rev{A continuous} weak solution to \labelcref{eq:p-poisson} in the sense of \cref{def:weak_solutions} is also a viscosity solution in the sense of \cref{def:visc_p}.
\end{lemma}
\begin{proof}
\rev{This statement can be found  in \cite[Proposition 4.8]{amato2021solutions}.
We remark that the full proof for the PDE on $\Omega$ can be found in \cite[Theorem 1.8]{medina2019viscosity}. 
The boundary conditions are treated in the precise same way as for the $p$-Laplacian eigenvalue problem (which can be regarded as a $p$-Poisson equation), see \cite[Lemma 2]{esposito2015neumann} and \cite[Proposition 3.2]{amato2021solutions}.}
\end{proof}

Before we turn to the limiting PDE we recall that the statement of \cref{prop:cvgc_p_to_infty}, which states that $\abs{\grad u_\infty}\leq 1$ almost everywhere in $\Omega$, can be converted into the viscosity framework.
\begin{proposition}\label{prop:grad_smaller_1}
If $\data_p\in\C(\closure{\Omega})$ converges uniformly to $\data\in\C(\closure{\Omega})$ as $p\to\infty$, then (up to a subsequence) viscosity solutions $u_p\in\W^{1,p}(\Omega)$ of \labelcref{eq:p-poisson} converge uniformly to a function $u_\infty\in\W^{1,\infty}(\Omega)$ which is a viscosity solution of $\abs{\grad u} -1 \leq 0$ and $1-\abs{\grad u}\geq 0$.
\end{proposition}
\begin{proof}
As in \cite[Proposition 4.7]{amato2021solutions}, which solely relies on \cite[Part III, Lemma 1.1]{bhattacharya1989limits}.
\end{proof}
It is important to remark that in the viscosity sense the inequality $\abs{\grad u}-1\leq 0$ is not equivalent to $1-\abs{\grad u}\geq 0$ which is why we make the distinction explicit.

Let us now turn to the limiting PDE \labelcref{eq:inf-poisson} satisfied by $u_\infty$ \rev{for which we assume that the limiting data $\data\in\C(\closure{\Omega})$ are continuous.} 
We prove that every limit $u_\infty$ of solutions to the $p$-Poisson equation \labelcref{eq:p-poisson} as $p\to\infty$ is a viscosity solution of \labelcref{eq:inf-poisson} which we restate here for convenience:
\begin{align}\label{eq:viscosity_PDE}
    \begin{cases}
    \min\left\lbrace\abs{\grad u}-1,-\Delta_\infty u\right\rbrace &= 0,\quad\text{in }\rev{\{\data>0\}},\\
    -\Delta_\infty u &= 0,\quad\text{in }\rev{\closure{\{\data\neq 0\}}^c},\\
    \max\left\lbrace1-\abs{\grad u},-\Delta_\infty u\right\rbrace &= 0,\quad\text{in }\rev{\{\data<0\}},\\
    %\frac{\partial u}{\partial\nu} &= 0,\quad\text{on }\partial\Omega,\\
    \max_{\closure{\Omega}} u + \min_{\closure{\Omega}} \rev{u} &=0.
    \end{cases}
\end{align}
\rev{
Note that this PDE does not contain any boundary conditions and it also does not specify the behavior on the closed set $\closure\Omega\setminus\left(\rev{\{\data>0\}}\cup\rev{\{\data<0\}}\cup\rev{\closure{\{\data\neq 0\}}^c}\right)$.
Note that even the weak boundary conditions in the viscosity sense, introduced before \cref{def:visc_p}, do not carry over to the limiting problem, which is consistent with the findings in \cite{amato2021solutions,bhattacharya1989limits}.
Regarding the behavior outside the three sets which occur in \labelcref{eq:viscosity_PDE} one should remark that the PDE is discontinuous there. 
Using lower and upper semi-continuous envelopes of this discontinuous function one can make sense of a weaker form of the PDE on the whole of $\Omega$, see \cite[Remark 4.3]{buccheri2021large} for a similar problem and \cite[Remark 6.3]{crandall1992user} for a general statement.}

In contrast, for Neumann eigenvalue problems of the $p$-Laplacian it is possible to formulate boundary conditions and obtain a limiting PDE on the whole of $\closure\Omega$, see \cite{esposito2015neumann}.

Let us now define what precisely we mean by viscosity solutions to the equation \labelcref{eq:viscosity_PDE}.
\begin{definition}[Viscosity solutions of the limiting equation]\label{def:visc_inf}
Let $\data\in\C(\closure\Omega)$.
An upper semi-continuous function $u:\closure\Omega\to\R$ is called a viscosity subsolution of \labelcref{eq:viscosity_PDE} if 
\begin{itemize}
    \item for all $x_0\in\Omega$ and $\phi\in C^2(\Omega)$ such that $u-\phi$ has a local maximum at $x_0$ it holds
    \begin{align*}
    \begin{cases}
    \min\left\lbrace\abs{\grad \phi(x_0)}-1,-\Delta_\infty \phi(x_0)\right\rbrace &\leq 0,\quad\text{if }x_0\in\rev{\{\data>0\}},\\
    -\Delta_\infty \phi(x_0) &\leq 0,\quad\text{if }x_0\in\rev{\closure{\{\data\neq 0\}}^c},\\
    \max\left\lbrace1-\abs{\grad \phi(x_0)},-\Delta_\infty \phi(x_0)\right\rbrace &\leq 0,\quad\text{if }x_0\in\rev{\{\data<0\}},
    \end{cases}
    \end{align*}
    % \item for all $x_0\in\partial\Omega$ and $\phi\in C^2(\closure\Omega)$ such that $u-\phi$ has a local maximum at $x_0$ it holds
    % \begin{align*}
    %     \min\left\lbrace \frac{\partial\phi}{\partial\nu}(x_0), 
    %     \min\left\lbrace\abs{\grad \phi(x_0)}-1,-\Delta_\infty \phi(x_0)\right\rbrace
    %     \right\rbrace \leq 0,
    % \end{align*}
    % \begin{align*}
    % \begin{cases}
    % \min\left\lbrace \frac{\partial\phi}{\partial\nu}(x_0), 
    % \min\left\lbrace\abs{\grad \phi(x_0)}-1,-\Delta_\infty \phi(x_0)\right\rbrace
    % \right\rbrace &\leq 0,\quad\text{if }\data(x_0)> 0,\\
    % \min\left\lbrace
    % \frac{\partial\phi}{\partial\nu}(x_0),
    % -\Delta_\infty\phi(x_0)
    % \right\rbrace
    % &\leq 0,\quad\text{if }\data(x_0)=0,\\
    % \min\left\lbrace \frac{\partial\phi}{\partial\nu}(x_0), 
    % \max\left\lbrace1-\abs{\grad \phi(x_0)},-\Delta_\infty \phi(x_0)\right\rbrace
    % \right\rbrace &\leq 0,\quad\text{if }\data(x_0)< 0,\\
    % \end{cases}
    % \end{align*}
    \item it holds $\max_{\closure\Omega} u + \essinf_\Omega u \leq 0$.
\end{itemize}
A lower semi-continuous function $u:\closure\Omega\to\R$ is called a viscosity supersolution if 
\begin{itemize}
    \item for all $x_0\in\Omega$ and $\phi\in C^2(\Omega)$ such that $u-\phi$ has a local minimum at $x_0$ it holds
    \begin{align*}
    \begin{cases}
    \min\left\lbrace\abs{\grad \phi(x_0)}-1,-\Delta_\infty \phi(x_0)\right\rbrace &\geq 0,\quad\text{if }x_0\in\rev{\{\data>0\}},\\
    -\Delta_\infty \phi(x_0) &\geq 0,\quad\text{if }x_0\in\rev{\closure{\{\data\neq 0\}}^c},\\
    \max\left\lbrace1-\abs{\grad \phi(x_0)},-\Delta_\infty \phi(x_0)\right\rbrace &\geq 0,\quad\text{if }x_0\in\rev{\{\data<0\}},
    \end{cases}
    \end{align*}
    % \item for all $x_0\in\partial\Omega$ and $\phi\in C^2(\closure\Omega)$ such that $u-\phi$ has a local minimum at $x_0$ it holds
    % \begin{align*}
    %     \max\left\lbrace \frac{\partial\phi}{\partial\nu}(x_0), 
    %     \max\left\lbrace 1-\abs{\grad \phi(x_0)},-\Delta_\infty \phi(x_0)\right\rbrace
    %     \right\rbrace \geq 0,
    % \end{align*}
    % \begin{align*}
    % \begin{cases}
    % \min\left\lbrace \frac{\partial\phi}{\partial\nu}(x_0), 
    % \min\left\lbrace\abs{\grad \phi(x_0)}-1,-\Delta_\infty \phi(x_0)\right\rbrace
    % \right\rbrace &\leq 0,\quad\text{if }\data(x_0)> 0,\\
    % \min\left\lbrace
    % \frac{\partial\phi}{\partial\nu}(x_0),
    % -\Delta_\infty\phi(x_0)
    % \right\rbrace
    % &\leq 0,\quad\text{if }\data(x_0)=0,\\
    % \min\left\lbrace \frac{\partial\phi}{\partial\nu}(x_0), 
    % \max\left\lbrace1-\abs{\grad \phi(x_0)},-\Delta_\infty \phi(x_0)\right\rbrace
    % \right\rbrace &\leq 0,\quad\text{if }\data(x_0)< 0,\\
    % \end{cases}
    % \end{align*}
    \item it holds $\esssup_\Omega u + \min_{\closure\Omega}u \geq 0$.
\end{itemize}
A function $u\in\C(\closure\Omega)$ is called viscosity solution it is both a sub- and a supersolution.
\end{definition}

Now we can prove the main theorem of this section.

\begin{proof}[Proof of \cref{thm:viscosity_PDE}]
The conditions of \cref{prop:cvgc_p_to_infty} are trivially fulfilled which guarantees the existence of a (subsequential) uniform limit $u_\infty\in\C(\closure\Omega)$.
We only show the subsolution property, the supersolution property can be shown analogously.
% \\
% \textbf{Interior points:}

Let $x_0\in\Omega$ and $\phi\in\C^2(\Omega)$ such that such that $u_\infty-\phi$ has a local maximum at $x_0$.
Choose a sequence $(p_i)_{i\in\N}\subset(d,\infty)$ converging to $\infty$ such that $u_{p_i}\to u_\infty$ uniformly.
Then there exists a sequence of points $(x_i)_{i\in\N}\subset \Omega$ converging to $x_0\in\Omega$ such that $u_{p_i}-\phi$ has a local maximum in $x_i$ for all $i\in\N$.
Since $u_{p_i}$ is a viscosity solution of \labelcref{eq:p-poisson}, by \labelcref{eq:decomposition} it holds
\begin{align}\label{ineq:p-viscosity}
    -\left(\abs{\grad\phi(x_i)}^{p_i-2}\Delta\phi(x_i) + (p_i-2)\abs{\grad\phi(x_i)}^{p_i-4}\Delta_\infty\phi(x_i) \right) = -\Delta_{p_i}\phi(x_i) \leq \data_{p_i}(x_i).
\end{align}
\textbf{Case 1, $x_0\in\rev{\{\data>0\}}$:}
\rev{
We have to show that
\begin{align}\label{eq:inequality_positive}
    \min\left\lbrace\abs{\grad\phi(x_0)}-1,-\Delta_\infty\phi(x_0)\right\rbrace \leq 0.
\end{align}
In fact, for showing this we will not even have to use that $\data(x_0)>0$ but \labelcref{eq:inequality_positive} is true for all $x\in\Omega$.
The condition $\data(x_0)>0$ will only be relevant for showing the converse inequality for supersolutions.}

If $\abs{\grad\phi(x_0)}=0$, we have by definition that $-\Delta_\infty\phi(x_0)=0$.
In the case that $\abs{\grad\phi(x_0)}>0$ we get for $i\in\N$ sufficiently large that $\abs{\grad\phi(x_i)}>0$ and can divide by this quantity to get:
\begin{align}\label{ineq:p-viscosity_divided}
    -\left(\frac{\abs{\grad\phi(x_i)}^{2}}{p_i-2}\Delta\phi(x_i) + \Delta_\infty\phi(x_i) \right) \leq \frac{1}{(p_i-2)\abs{\grad\phi(x_i)}^{p_i-4}}\data_{p_i}(x_i).
\end{align}
Then either $\abs{\grad\phi(x_0)}-1\leq 0$ or $\abs{\grad\phi(x_0)}-1 > 0$, \rev{and in the latter} case we obtain, $-\Delta_\infty\phi(x_0)\leq 0$ by taking the limit $i\to\infty$ and using uniform boundedness of $\data_{p_i}$.
Combining all those cases yields \labelcref{eq:inequality_positive}.
\\
\textbf{Case 2, $x_0\in\rev{\closure{\{\data\neq 0\}}^c}$:}
\rev{We have to show that
\begin{align}\label{eq:inequality_zero}
    -\Delta_\infty\phi(x_0)\leq 0.
\end{align}}
If $\abs{\grad\phi(x_0)}=0$, then we have by definition that $-\Delta_\infty\phi(x_0)=0$. 
If $\abs{\grad\phi(x_0)}>0$ then \labelcref{ineq:p-viscosity_divided}, the openness of $\rev{\closure{\{\data\neq 0\}}^c}$, and the uniform convergence of $\data_{p_i}$ imply that $-\Delta_\infty \phi(x_0)\leq 0$.
\rev{In either case, we obtain \labelcref{eq:inequality_zero}.}
\\
\textbf{Case 3, $x_0\in\rev{\{\data<0\}}$:}
\rev{We have to show that
\begin{align}\label{eq:inequality_negative}
    \max\left\lbrace 1-\abs{\grad\phi(x_0)},-\Delta_\infty\phi(x_0)\right\rbrace \leq 0.
\end{align}}
Since the functions $\data_{p_i}$ converge uniformly to $\data$, the set $\rev{\{\data<0\}}$ is open, and $\data(x_0)$ is strictly negative, \labelcref{ineq:p-viscosity} can only be satisfied if $\abs{\grad\phi(x_i)}>0$ for $i\in\N$ sufficiently large.
Dividing by $(p_i-2)\abs{\grad\phi(x_i)}^{p_i-4}>0$ then again yields \labelcref{ineq:p-viscosity_divided}.
If $1-\abs{\grad\phi(x_0)}\leq 0$, then in the limit $i\to\infty$ and using uniform boundedness of $\data_{p_i}$ we get $-\Delta_\infty \phi(x_0) \leq 0$
If, however, $1-\abs{\grad\phi(x_0)}>0$, one gets $-\Delta_\infty\phi(x_0)\leq -\infty$ as $i\to\infty$ which is impossible since $\phi\in\C^2(\Omega)$.
Combining these two cases we \rev{obtain \labelcref{eq:inequality_negative}.}
% \\
% \textbf{Boundary points:}
% Let us now turn to the boundary conditions: fix $x_0\in\partial\Omega$ and $\phi\in C^2(\closure\Omega)$ such that $u_\infty-\phi$ has a local maximum at $x_0$.
% As before there are points $(x_i)_{i\in\N}\subset\closure\Omega$ converging to $x_0$ such that $u_{p_i}-\phi$ has a local maximum at $x_i$.
% If for infinitely many $i\in\N$ the points $x_i$ lie in $\Omega$, we are back in the previous situation.
% In particular, using Case 1 above yields
% \begin{align*}
%     \min\left\lbrace
%     \abs{\grad\phi(x_0)}-1,-\Delta_\infty\phi(x_0)
%     \right\rbrace
%     \leq 0
% \end{align*}
% independent of the sign of $\data(x_0)$.
% If infinitely many points $x_i$ lie in $\partial\Omega$, we get for these $i\in\N$ by \cref{def:visc_p} that
% \begin{align*}
%     \min\left\lbrace
%     \abs{\grad\phi(x_i)}^{p_i-2}\frac{\partial\phi(x_i)}{\partial\nu},
%     -\Delta_{p_i}\phi(x_i)-\data_{p_i}(x_i)
%     \right\rbrace
%     \leq 0.
% \end{align*}
% For the second term in this minimum we use again Case 1.
% For the first one, we argue as follows:
% If $\abs{\grad\phi(x_0)}=0$ then $\frac{\partial\phi}{\partial\nu}(x_0)=0$.
% Otherwise, we can divide by $\abs{\grad\phi(x_i)}^{p_i-2}$ for $i\in\N$ sufficiently large and get that $\frac{\partial\phi}{\partial\nu}(x_i)\leq 0$.
% For $i\to\infty$ we use continuity of $\grad\phi$ and get $\frac{\partial\phi}{\partial\nu}(x_0)\leq 0$ which shows the assertion.
\\
\textbf{Mean value:}
Finally, let us turn to the mean value condition. 
Letting $u^\pm := \max(\pm u,0)$ denote the positive/negative part of a function $u:\Omega\to\R$, for all $i\in\N$ it holds
\begin{align*}
    0 \geq \int_\Omega \abs{u_{p_i}}^{p_i-2}u_{p_i}\dx 
    =
    \int_\Omega \abs{u_{p_i}^+}^{p_i-2}u_{p_i}^+\d x 
    - \int_\Omega \abs{u_{p_i}^-}^{p_i-2}u_{p_i}^-\d x 
\end{align*}
which is equivalent to
\begin{align*}
    \int_\Omega\abs{u_{p_i}^+}^{p_i-1} \d x \leq \int_\Omega\abs{u_{p_i}^-}^{p_i-1} \d x.
\end{align*}
Applying \cref{lem:convergence_norms} with $k=1$ then yields
\begin{align*}
    \esssup_\Omega\abs{u_\infty^+} \leq \esssup_\Omega\abs{u_\infty^-},
\end{align*}
which by the upper semi-continuity of $u_\infty$ is equivalent to $\max_{\closure\Omega} u_\infty + \essinf_\Omega u_\infty \leq 0$.
\end{proof}
Next we show that viscosity solutions of \labelcref{eq:viscosity_PDE} have the intriguing property that they solve the eikonal equation $\abs{\grad u}=1$ on the interior of the support of the data $\data$ and are infinity harmonic elsewhere.
But even more is true: namely that $u$ is an infinity superharmonic solution of the eikonal equation on $\rev{\{\data>0\}}$ and analogously a subharmonic one on $\rev{\{\data<0\}}$.
This is formalized in the following corollary.
\begin{corollary}
The function $u_\infty$ is a viscosity solution of
\begin{align}
    \begin{cases}
    \abs{\grad u} - 1=0,\quad-\Delta_\infty u\geq 0\quad&\text{in }\rev{\{\data>0\}},\\
    -\Delta_\infty u = 0,\quad&\text{in }\rev{\closure{\{\data\neq 0\}}^c},\\
    1-\abs{\grad u} =0,\quad-\Delta_\infty u\leq 0\quad&\text{in }\rev{\{\data<0\}}.
    \end{cases}
\end{align}
\end{corollary}
\begin{proof}
Let $x_0\in\rev{\{\data>0\}}$ and $\phi\in\C^2(\Omega)$ such that $u_\infty-\phi$ has a local minimum at $x_0$.
Since $u_\infty$ is in particular a supersolution of \labelcref{eq:viscosity_PDE} it follows that
\begin{align*}
    \abs{\grad\phi(x_0)}-1\geq 0\quad\text{and}\quad-\Delta_\infty\phi(x_0)\geq 0.
\end{align*}
Furthermore, \cref{prop:grad_smaller_1} shows that $\abs{\grad u_\infty}-1\leq 0$ in the viscosity sense which shows the claim.
If $x_0\in\rev{\{\data<0\}}$ one analogously uses the subsolution property of $u_\infty$ to infer that $1-\abs{\grad u_\infty}\leq 0$ and $-\Delta_\infty u_\infty\leq 0$ in the viscosity sense and again utilizes \cref{prop:grad_smaller_1} to conclude.
\end{proof}
It is important to remark that the limiting PDE \labelcref{eq:viscosity_PDE} does not admit unique solutions.
This is illustrated in the following example.
\begin{example}[Non-uniqueness of the limiting PDE]
Let us consider the situation $\Omega=(-1,1)\subset\R$ and $\data$ being an arbitrary continuous function with $\{\data<0\}=(-1,0)$ and $\{\data>0\}=(0,1)$.
We claim that the following family of functions (see \cref{fig:1d_sols})
\begin{align*}
    u_t(x) = 
    \begin{cases}
    \abs{x+t}-t, &x\in[-1,-0.5],\\
    x, &x\in(-0.5,0.5),\\
    t-\abs{x-t}, &x\in[0.5,1],
    \end{cases}
    \qquad
    t\in[0.5,1],
\end{align*}
is a viscosity solution of \labelcref{eq:viscosity_PDE}.
Indeed, it is trivial to see that $u_t$ is even a classical solution of \labelcref{eq:viscosity_PDE} on $(-2,2)\setminus\{\pm t\}$.
So we just have to check the two corner points at $\pm t$.
For $x_0=-t$ and $\phi\in\C^2(\Omega)$ touching $u$ from above in $x_0$ it is obvious that $\abs{\phi'(x_0)}\leq 1$ and hence $\min\left\lbrace\abs{\phi(x_0)}-1,-\Delta_\infty\phi(x_0)\right\rbrace\leq 0$.
Furthermore, there is no $\phi\in\C^2(\Omega)$ touching $u$ from below in $x_0$. 
Similarly, one can argue for $x_0=t$ and obtain that $u_t$ is a viscosity solution of \labelcref{eq:viscosity_PDE}.
Note that none of the functions $u_t$ has homogeneous Neumann boundary conditions.
\begin{figure}[thb]
    \centering
    \begin{tikzpicture}
    \begin{axis}[
    axis lines = center,
    axis equal image,
    xtick = {-1,-0.5,0,0.5,1},
    xticklabels = {$-1$,$-0.5$,$0$,$0.5$,$1$},
    ytick = {-1,-0.5,0,0.5,1},
    yticklabels = {$-1$,$-0.5$,$0$,$0.5$,$1$}
    ]

\def\function(#1,#2){
+(#1 < -0.5) * (-(#2)+abs((#1)+(#2))) 
+and(#1 >= -0.5, #1 <= 0.5) * (#1)
+(#1 > 0.5) * ((#2)-abs((#1)-(#2)))
}

\foreach [evaluate=\t as \n using ln(2*\t)/ln(2) * 100] \t in {0.5,0.625,0.75,0.875,1}%{1,1.25,1.5,1.75,2}
{
\edef\temp{%
\noexpand
\addplot [
    domain=-1:1, 
    samples=151, 
    color=blue!\n!red,
    smooth
]
{\function(x,\t)};
}\temp
}
\end{axis}
\end{tikzpicture}
\caption{Five different solutions $u_t$ of the limiting PDE \labelcref{eq:viscosity_PDE} for $t\in\{0.5,0.625,0.75,0.875,1\}$ (from red to blue). The linear function $u(x)=x$ is a limit of $p$-Laplacian solutions, see \cref{ex:p-laplacians}.}
\label{fig:1d_sols}
\end{figure}
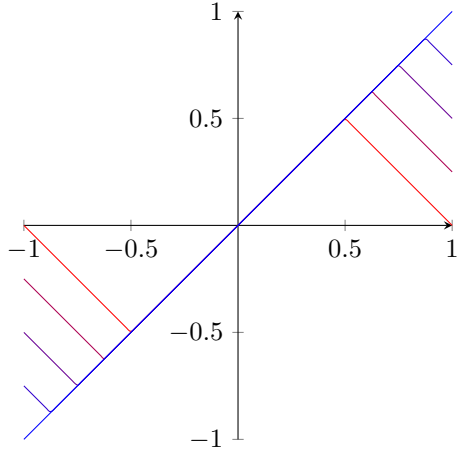
\end{example}
Since the concept of viscosity solutions heavily relies on continuity and is not compatible with discontinuous or even measure data $\data$, we have to use approximation techniques if we want to make sense of \labelcref{eq:viscosity_PDE} if $\data$ is a measure.
In particular, it seems natural to replace the open set $\rev{\closure{\{\data\neq 0\}}^c}$ by the open set $\Omega\setminus\supp\data$.
However, one cannot just replace $\{\data \gtrless\}$ by $\supp\data^\pm$ since the latter sets are not open and might even have empty interior.
For an arbitrary measure data $\data\in\M(\closure{\Omega})$, which we extend to zero outside $\closure\Omega$, we consider the mollifications 
\begin{align}\label{eq:convolution}
    \data^\eps(x)=
    \int_{\Omega}\phi_\eps(x-y)\d\data(y),\quad x\in\R^d,
\end{align}
where $\phi\in\C^\infty_c(\R^d)$ is a smooth kernel with $\supp\phi\in B_1(0)$ and $\phi_\eps(x)=\eps^{-d}\phi(x/\eps)$.
It is obvious from the definition of $\data_\eps$ that if $x\in\Omega\setminus\supp\data$ then $x\in\Omega\setminus\supp\data^\eps$ for all $\eps>0$ small enough.
Furthermore $\data^\eps\strictto\data$ as $\eps\downarrow 0$.
Using the techniques from the proof of \cref{thm:viscosity_PDE} we immediately get the following result.
\begin{corollary}
Let $\data\in\M(\closure{\Omega})$ and $\data_p:=\data^{\eps_p}\restr\closure\Omega\in\C(\closure\Omega)$ where $\lim_{p\to\infty}\eps_p=0$ and $\data^{\eps_p}$ is defined as in \labelcref{eq:convolution}.
Let furthermore $u_p\in\W^{1,p}(\Omega)$ be viscosity solutions of \labelcref{eq:p-poisson} with data $\data_p\in\C(\closure{\Omega})$.
Then the function $u_\infty\in\W^{1,\infty}(\Omega)$ is a viscosity solution of
\begin{align}
    -\Delta_\infty u = 0 \quad\text{in }\Omega\setminus\supp\data.
\end{align}
\end{corollary}
\begin{proof}
Let $x_0\in\Omega\setminus\supp\data$ and $\phi\in\C^2(\Omega)$ such that such that $u_\infty-\phi$ has a local maximum at $x_0$.
Choose a sequence $(p_i)_{i\in\N}\subset(d,\infty)$ converging to $\infty$ such that $u_{p_i}\to u_\infty$ uniformly.
As always there exists a sequence of points $(x_i)_{i\in\N}\subset \Omega$ converging to $x_0\in\Omega$ such that $u_{p_i}-\phi$ has a local maximum in $x_i$ for all $i\in\N$.
For all sufficiently large $i\in\N$ it holds $x_0\in\Omega\setminus\supp\data_{p_i}$ and hence $\data_{p_i}(x_0)=0$.
As in Case 2 of the proof of \cref{thm:viscosity_PDE} we can conclude that $-\Delta\phi(x_0)\leq 0$.
The supersolution property is shown analogously.
\end{proof}

\section*{Conclusion}

In this article we have investigated limits of the $p$-Laplace equation with measure-valued right hand side as $p\to\infty$. 
We proved existence of (subsequential) limits and characterized them as Kantorovich potentials for the optimal transport problem of transporting the positive part of the right hand side onto the negative one. 
For continuous data, we also proved that such limits are viscosity solutions of a degenerate PDE, involving the infinity Laplacian and the eikonal equation.
It will be interesting to investigate in which sense the limiting PDE can be interpreted for measure-valued data which has a support with empty interior.
Here, lower / upper semi-continuous relaxations as in \cite{crandall1992user,buccheri2021large} might be promising tools.

\section*{Availability of data and materials}%% if any

Not applicable.

\section*{Competing interests}

The author declares that he has no competing interests.

\section*{Funding}%% if any

This work was supported by the Deutsche Forschungsgemeinschaft (DFG, German Research Foundation) under Germany's Excellence Strategy - GZ 2047/1, Projekt-ID 390685813. 
Parts of this work were also done while the author was in residence at Institut Mittag-Leffler in Djursholm, Sweden during the semester on \textit{Geometric Aspects of Nonlinear Partial Differential Equations} in 2022, supported by the Swedish Research Council under grant no. 2016-06596.

\section*{Authors' contributions}

Not applicable since this is a single author publication.

\section*{Acknowledgements}

The author would like to thank Jeff Calder and Simone Di Marino for enlightening discussions.
This work was partially done while the author was visiting the Simons Institute for the Theory of Computing to participate in the program ``Geometric Methods in Optimization and Sampling'' during the Fall of 2021 and the author is very grateful for the hospitality of the institute.

\printbibliography

\end{document}